\documentclass[12pt,reqno]{article}

\usepackage[usenames]{color}
\usepackage{color}
\usepackage{amssymb}
\usepackage{graphicx}
\usepackage{amscd}
\usepackage{psfrag}
\usepackage[colorlinks=true,
linkcolor=webgreen, filecolor=webbrown,
citecolor=webgreen]{hyperref}

\definecolor{webgreen}{rgb}{0,.5,0}
\definecolor{webbrown}{rgb}{.6,0,0}

\usepackage{color}
\usepackage{fullpage}
\usepackage{float}

\usepackage{graphics,amsmath}
\usepackage{amsfonts,amsthm}
\usepackage{latexsym}
\usepackage{epsf}

\newcommand{\seqnum}[1]{\href{http://www.research.att.com/cgi-bin/access.cgi/as/~njas/sequences/eisA.cgi?Anum=#1}
{\underline{#1}}}

\begin{document}
\begin{center}
\vskip 1cm{\Large \bf On the enumeration of three-rowed 
standard Young tableaux of skew shape 
in terms of Motzkin numbers } \vskip 1cm
\large Jong Hyun Kim \\
Department of Mathematics\\
Brandeis University\\
Waltham,  MA 02454-9110 \\
USA \\
\href{mailto:jhkim@brandeis.edu}{\tt jhkim@brandeis.edu} \\
\end{center}

\vskip .2 in
\begin{abstract}
 The enumeration of standard Young
tableaux (SYTs) of shape $\lambda$ 
can be easily
computed by the hook-length formula. 
In $1981$, Amitai Regev proved that the number of SYTs having at most three rows with $n$ entries equals the $n$th Motzkin number $M_n$.
In $2006$, Regev conjectured that the total number of SYTs of skew shape $\lambda/(2,1)$ over all partitions $\lambda$ having at most three parts with $n$ entries is the difference of two Motzkin numbers, $M_{n-1}-M_{n-3}$. Ekhad and Zeilberger proved Regev's conjecture using a computer program.
In his paper \cite{E}, S.-P. Eu found a bijection between Motzkin paths and SYTs of skew shape with at most three rows to prove Regev's conjecture, and Eu also indirectly showed that for the fixed $\mu=(\mu_1,\mu_2)$ the number of SYTs of skew shape $\lambda/ \mu $ over all partitions $\lambda$ having at most three parts can be expressed as a linear combination of the Motzkin numbers. 
In this paper, we will find an explicit formula for the generating function for the general case: for each partition $\mu$ having at most three parts the generating function gives a formula for the coefficients of the linear combination of Motzkin numbers. We will also show that these generating functions are unexpectedly related to the Chebyshev polynomials of the second kind.
\end{abstract}

\newtheorem{theorem}{Theorem}[section]
\newtheorem{lemma}{Lemma}[section]

\section{Introduction}

A {\it partition} $\lambda$ of a nonnegative integer $n$ is a weakly decreasing sequence
\[\lambda=(\lambda_1, \lambda_2, \ldots, \lambda_l)\] such that $\sum_{i=1}^l \lambda_i=n$ and $\lambda_l >0$. The integers $\lambda_i$ are called {\it parts}.
The number of parts of $\lambda$ is the {\it length} of
$\lambda$, denoted $l(\lambda)$. We denote Par$(n)$ to be the
set of all partitions of $n$, with Par$(0)$ consisting of the
empty partition $\emptyset$, and we let
\[\textrm{Par} := \cup_{n \ge 0} \textrm{ Par}(n).\] If $\lambda \in
\textrm{Par}(n)$, then we also write $\lambda \vdash n$ or $|\lambda|=n$.
Any partition $\lambda$ can be identified with its {\it Young diagram}, which is a collection of boxes arranged in
left-justified rows with the $i$th row containing $\lambda_i$ boxes for $1 \le i \le l(\lambda)$. If $\lambda$ and $\mu$ are partitions such that $\mu_i \le \lambda_i$ for all $i$, we write $\mu \subset \lambda$. Let us assume $|\lambda|=n$ and $\mu \subset \lambda$ throughout this paper.

A {\it skew shape} $\lambda / \mu$ is a pair of partitions $(\lambda,\mu)$ such that the Young diagram of $\lambda$ contains the Young diagram of $\mu$. 
If $\mu = \emptyset$, then we assume that $\lambda / \mu = \lambda$.
A {\it skew Young diagram} of shape $\lambda/ \mu$ is a Young diagram of $\lambda$ with a Young diagram of $\mu$ removed from it.
A {\it standard Young tableau} (SYT) of skew shape $\lambda / \mu$ is obtained by taking a skew Young diagram of shape $\lambda / \mu$ and writing numbers $1, 2,\ldots, n$ in the $n$ boxes of this diagram such that the numbers
increase from left to right in each row and from top to bottom
down in each column. A Young tableau is called {\it semistandard} (SSYT) if the entries weakly increase from left to right in each row and strictly increase from top to bottom down in each column. The size of a SSYT is the number of its entries.
As shown in Figure below, $(a)$ is a Young diagram of shape $(4,3,1)$, $(b)$ is a SYT of shape $(4,3,1)$, and $(c)$ is a SYT of skew shape $(4,3,1)/(2,1)$, respectively.
\begin{tabbing}
a\qquad\qquad\qquad \qquad\qquad\qquad\=a \qquad \qquad\qquad\qquad\qquad\qquad\= \kill
 {\includegraphics[width=1.3in]{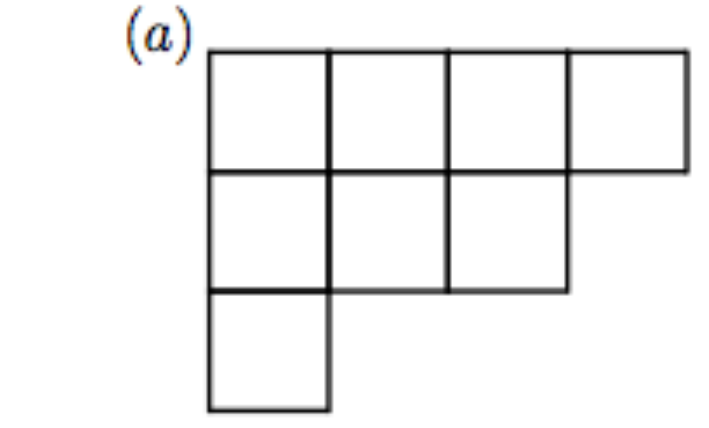}} \>
 {\includegraphics[width=1.3in]{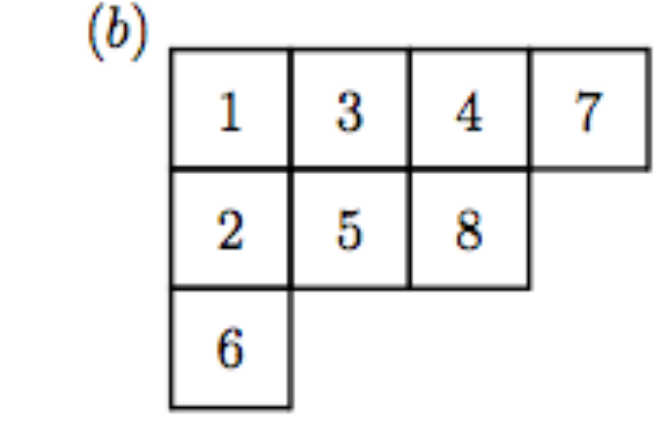}}   \>
 {\includegraphics[width=1.3in]{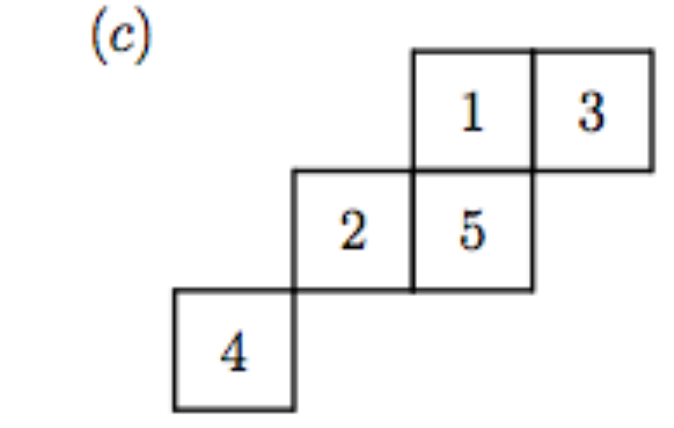}}  \\
\end{tabbing}
\vspace{-.5cm}

Now let us review several definitions for symmetric functions \cite{EC2}. Let $x=(x_1,x_2,\ldots)$ be a set of indeterminates. For $n \in \mathbb{N}$, a {\it homogeneous
symmetric function of degree} $n$ over the complex number field
$\mathbb{C}$ is a formal power series
\[f(x)=\sum_\alpha c_\alpha x^\alpha, \] where $\alpha$ ranges
over all weak compositions $\alpha=(\alpha_1,\alpha_2, \ldots)$ of
$n$, $c_\alpha \in \mathbb{C}$, $x^\alpha$ stands for the monomial
$x_1^{\alpha_1}x_2^{\alpha_2}\cdots$, and
$f(x_{\omega(1)},x_{\omega(2)},\ldots)=f(x_1,x_2,\ldots)$ for
every permutation $\omega$ of the positive integers $\mathbb{P}$.
Note that a symmetric function of degree zero is just a complex
number.

 The {\it monomial symmetric functions} $m_\lambda$ for
$\lambda \in $ Par are defined by
\[m_\lambda := \sum_\alpha x^\alpha,\]
where the sum ranges over all distinct permutations
$\alpha=(\alpha_1,\alpha_2,\ldots)$ of the entries of the vector
$\lambda=(\alpha_1,\alpha_2,\ldots)$.

The {\it complete homogeneous symmetric functions} $h_\lambda$ for
$\lambda \in $ Par are defined by the formulas
\[h_n := \sum_{\lambda \vdash n} m_\lambda = \sum_{i_1 \le \cdots \le i_n}x_{i_1}\cdots x_{i_n}, \textrm{ for } n \ge 1
\qquad (\textrm{with }h_0=1)\]
\[h_\lambda = h_{\lambda_1} h_{\lambda_2} \cdots \qquad \textrm{ if } \lambda=(\lambda_1,\lambda_2,\ldots ).\]

Let $\lambda / \mu$ be a skew shape. The {\it skew Schur function
$s_{\lambda / \mu}=s_{\lambda / \mu}(x)$ of $\lambda / \mu$}
in the variables $x=(x_1,x_2, \ldots)$ is the sum of monomials
\[s_{\lambda / \mu}(x):=\sum_{T}x^T=\sum_{T}x_1^{t_1}x_2^{t_2}\cdots ,\] summed over all SSYTs $T$ of skew shape $\lambda / \mu$ where each $t_i$ counts the occurrences of the number $i$ in the SSYT $T$. If $\mu = \emptyset$, (that is, $\lambda / \mu = \lambda$), then we call $s_\lambda(x)$ the
Schur function of shape $\lambda$. For example, the SSYTs of shape $(2,1)$ with largest part at most three are given by
\[{{\includegraphics[width=5in]{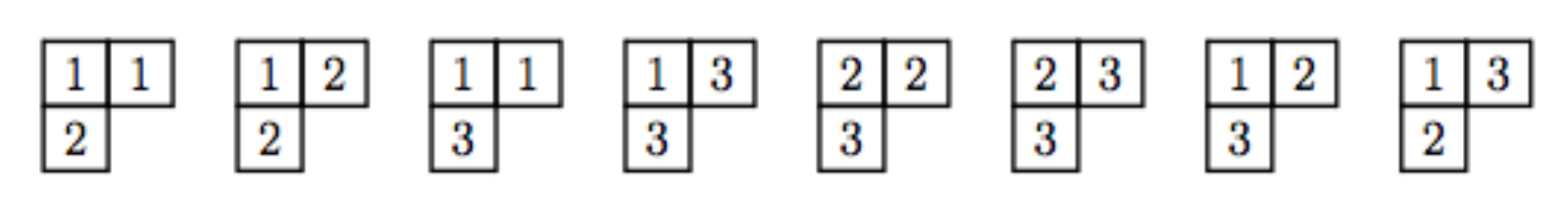}}}.\]
So, we have
        \begin{align*}S_{(2,1)}(x_1,x_2,x_3)&=x_1^2x_2+x_1x_2^2+x_1^2x_3+x_1x_3^2+x_2^2x_3+x_2x_3^2+2x_1x_2x_3\\
        &=m_{(2,1)}(x_1,x_2,x_3)+2m_{(1,1,1)}(x_1,x_2,x_3).
                \end{align*}

\section{About standard Young tableaux having at most three rows}
In their paper \cite{GH}, Gordon and Houten showed that the sum of Schur functions of skew shape $\lambda / \mu$, for fixed $\mu$ over all partitions $\lambda$ with at most a fixed
number of rows, can be expressed as a Pfaffian of a matrix of
complete homogeneous symmetric functions. Before we state their
results, let us define the functions $h$, $g_i$, and $f_j$ used in
them.

For $i \ge 0$ and $ j \ge 1$, let us define $h$, $g_i$, and $f_j$ by
{\allowdisplaybreaks
        \begin{align*}
            h &:= \sum_{n=0}^\infty h_n  = \sum_{n=0}^\infty \sum_{\lambda \vdash n} m_\lambda ,\\
            g_i &:= \sum_{n=0}^\infty h_n h_{n+i} = \sum_{n=0}^\infty \sum_{\lambda_1 \vdash n} m_{\lambda_1} \sum_{\lambda_2 \vdash n+i} m_{\lambda_2},\\
            \intertext{and}
            f_j &:= g_0 +2(g_1 + \cdots + g_{j-1}) + g_j .
        \end{align*}
        }
Gordon and Houten proved that
\begin{equation*}
       \sum_{l( \lambda )  \le 2m} s_{\lambda / \mu} = \textrm{Pf} (D_{2m}),
           \end{equation*}
where Pf denotes the Pfaffian and $D_{2m}=( f_{\lambda_i -\mu_j+j-i}) _{1 \leq i,j \leq 2m}$, and     
\begin{equation} \label{eq:721}
        \sum_{l( \lambda )  \le 2m+1} s_{\lambda / \mu}  = \textrm{Pf}\left(%
            \begin{array}{cc}
            0 & H \\
            -H^t & D_{2m} \\
            \end{array}%
            \right),
    \end{equation}
where the matrix is obtained by bordering $D_{2m}$
 with a row $H$ of $h$'s, a column $-H^t$
(transpose of $-H$) of $-h$'s, and a zero.

For $\mu = \emptyset$, Gordon
\cite{G} showed that
        \begin{equation} \label{eq:722}
        \sum_{l( \lambda )  \le 2m} s_{\lambda} =
            \det (g_{i-j}+g_{i+j-1})_{1 \le i,j \le m},
        \end{equation}
and
    \begin{equation}\label{eq:723}
        \sum_{l( \lambda )  \le 2m+1} s_{\lambda }  =
            h \det (g_{i-j}-g_{i+j})_{1 \le i,j \le m}.
    \end{equation}

For the case $m=1$, identity $(\ref{eq:722})$ reduces
to
        \begin{equation} \label{eq:724}
        \sum_{l( \lambda ) \le 2} s_{\lambda} =
            g_0 + g_1.
        \end{equation}
We can directly prove identity $(\ref{eq:724})$ by considering
the two cases of partitions of even or odd numbers and then by applying
Pieri's rule to each case. For $m=1$, identity $(\ref{eq:723})$
reduces to
        \begin{equation*}
        \sum_{l( \lambda ) \le 3} s_{\lambda} =
           h(g_0 - g_2),
        \end{equation*}
which can also be proved by applying Pieri's rule to expand $hg_0$ and $hg_2$.

To count SYTs of skew shape $\lambda / \mu$,
we need to find the coefficient of $x_1x_2\cdots x_{n-|\mu|}$ in $s_{\lambda / \mu}$. Now let us transform symmetric functions
into formal power series by applying the map $\theta$ from the algebra of symmetric
functions to formal power series in $x$, which is defined by
        \begin{align*}
            \theta(m_{\lambda}) &=
         \left\{
        \begin{array}{l l}
        x^r/r! , &\, \textrm{if} \,\, \lambda=(1^r) \textrm{ for some $r$}\\
        0, & \,  \textrm{otherwise} \\
        \end{array} \right.
        \end{align*} and extended by linearity.
Then we can easily show that the map $\theta$ is a homomorphism with the property
\[          \theta(h_n) = \frac{x^n}{n!} . \]
Let us apply the linear map $\theta$ to $g_i$. For $i \ge 0$, we have
        \begin{align*}
            \theta(g_i) &= \theta \Big(\sum_{n=0}^\infty h_n h_{n+i} \Big)
            \\
            &= \sum_{n=0}^\infty \theta(h_n) \theta(h_{n+i})
            \\
            &= \sum_{n=0}^\infty {x^n \over n!}{x^{n+i} \over (n+i)! }
            \\
            &= \sum_{n=0}^\infty  {2n +i \choose n}{x^{2n+i} \over (2n+i)! }.        
        \end{align*}

Now we want to convert the generating function $\theta(g_i)$ to an ordinary
generating function. To do this, we need the following lemma.

\begin{lemma} \label{lemma:72}
Define $L$ to be the linear map from the
algebra of formal power series $\mathbb{C}[[x]]$ to itself defined
by $\frac{x^n}{n!} \mapsto x^n$, extended by linearity. Then we have
\[ L(e^x f(x))={1 \over 1-x} F\Big({x \over 1-x}\Big), \textrm{ where }
F(x)=L\big(f(x)\big). \]
\end {lemma}

\begin{proof}
Let $f(x)=\sum_{n=0}^\infty a_n {x^n \over n!}$. Then $F(x)=L\big(f(x)\big)
=\sum_{n=0}^\infty a_n x^n $ and
    \begin{align*}
    L(e^x f(x)) 
     &=  \sum_{n=0}^\infty a_n L\Big(e^x{x^n \over n!} \Big)\quad \textrm{by linearity}\\
     &=  \sum_{n=0}^\infty a_n L\Big(\sum_{k=0}^\infty {n+k \choose k} {x^{n+k} \over (n+k)!} \Big)\\
     &=  \sum_{n=0}^\infty a_n x^n \sum_{k=0}^\infty {n+k \choose k} x^{k} \\
     &= \sum_{n=0}^\infty a_n {x^n \over (1-x)^{n+1}} \\
     &= {1 \over 1-x} F\Big({x \over 1-x}\Big).\qedhere
    \end{align*} 
\end{proof}

Let $c(x)$ be the Catalan number generating function
\[ c(x)=\sum_{n=0}^\infty C_n x^n=\sum_{n=0}^\infty \frac1{n+1}\binom{2n}{n} x^n=\frac{1-\sqrt{1-4x}}{2x},\]
and $m(x)$ be the Motzkin number generating function
\[ \displaystyle m(x)=\sum_{n=0}^\infty M_n x^n = {1-x-\sqrt{1-2x-3x^2} \over 2x^2},\] where $M_n=\sum_{k=0}^{\left\lfloor n/2 \right\rfloor}{n !}/\big({k!(k+1)!(n-2k)!}\big)$ is the $n$th Motzkin number.
From the well-known fact
    \begin{equation} \label{eq:725}
        \sum_{n=0}^\infty  {2n + s \choose n} x^{n} = {c(x)^s \over \sqrt{1-4x}},
    \end{equation} we
can express the Motzkin number generating function in terms of the
Catalan number generating function:
\[ {1 \over 1-x} c \left( {x^2 \over
(1-x)^2} \right) =m(x). \]

Let $\Psi :=L \circ \theta $. Then $\Psi$ is a linear map from the algebra of symmetric functions to the algebra of formal power series $\mathbb{C}[[x]]$.

Now let us compute $\Psi(g_i)$ for $i \ge 0$.
        \begin{align} \label{eq:726}
       \Psi(g_i) &= L(\theta(g_i) )  \nonumber \\
             &= L\left(  \sum_{n=0}^\infty
            {2n +i \choose n}{x^{2n+i} \over (2n+i)! } \right)
            \nonumber \\
            &= \sum_{n=0}^\infty  {2n +i \choose n} x^{2n+i}
            \nonumber \\
            &= {x^ic(x^2)^i \over \sqrt{1-4x^2}},
        \end{align}
where the last equation follows from identity $(\ref{eq:725})$.

For $j, k \ge 1$, let us consider a partition $\mu=(\mu_1,\mu_2)$ having $\mu_1= j+k-2$ and $\mu_2= j-1$ in equation $(\ref{eq:721})$.
Then equation $(\ref{eq:721})$ reduces to
    \begin{equation}  \label{eq:727}
     \displaystyle \sum_{ l(\lambda) \le 3} s_{\lambda/\mu} =h(f_j+f_k-f_{j+k}).
     \end{equation}

Let $G_3(\mu_1, \mu_2)$ be the sum $(\ref{eq:727})$. That is, 
    \begin{equation}  \label{eq:728}
     G_3(j+k-2,j-1) =h(f_j+f_k-f_{j+k}).
     \end{equation}

Now let us first consider the case $\mu=(k-1,0)$, that is, $\mu=(k-1)$.
Letting $j=1$ in equation $(\ref{eq:728})$, we can find the
number of SYTs of
skew shape $\lambda / (k-1)$ over all partitions $\lambda$ of $n$ having at most three parts, filled with the numbers
$1,2,\ldots,n-k+1$. Since
\[f_n =g_0 +2(g_1+\cdots +g_{n-1})+g_n \, \textrm{ for } n>0, \]
we have \[ G_3(k-1,0)=h(g_0+g_1-g_k-g_{k+1}). \]

           By equation $(\ref{eq:726})$ and Lemma $\ref{lemma:72}$, we have
{\allowdisplaybreaks
        \begin{align} \label{eq:729}
            \Psi(h g_k) &= L \big( \theta(h g_k) \big) \nonumber \\
             &= L \big( \theta(h) \theta(g_k) \big) \nonumber \\
             &= L \big( e^x \theta(g_k) \big) \nonumber \\
            &= \Big({x \over 1-x} \Big )^k c\Big({x^2 \over (1-x)^2}\Big)^k {1 \over \sqrt{1-2x-3x^2}}\nonumber \\
            &= {x^km(x)^k \over \sqrt{1-2x-3x^2}}.
         \end{align}
}

To simplify the term that we computed above, let us define $\alpha_k$ by
\[ \alpha_k:= \Psi(hg_k), \textrm{ for all } k \ge 0.\] Then we have $\alpha_k= x^km(x)^k / \sqrt{1-2x-3x^2}$.

To find the number of SYTs of skew shape $\lambda/(k-1)$ over all partitions $\lambda$ of $n$ having at most three parts, filled with the numbers $1,2,\ldots,n-k+1$, let us apply $\Psi$ to $G_3(k-1,0)$.
Then by linearity we deduce
\begin{align}\label{eq:730}
           \Psi \big( G_3(k-1,0)\big) &=
           \alpha_0 +\alpha_1 -\alpha_k -\alpha_{k+1} \nonumber  \\
            &= {1+xm(x)-x^km(x)^k-x^{k+1}m(x)^{k+1} \over \sqrt{1-2x-3x^2}}\nonumber \\
            &= {\big( 1+xm(x)\big)\big(1-x^km(x)^k\big) \over \sqrt{1-2x-3x^2}} \nonumber \\
            &= {m(x) \over 1-x^2m(x)^2}  \big( 1+xm(x)\big)\big(1-x^km(x)^k\big) \nonumber \\
            &= {m(x) \over 1-xm(x)}  (1-x^km(x)^k ).
         \end{align}

For example, consider the case $k=1$ in $(\ref{eq:730})$, that is, $\mu= \emptyset$. Then we know that the number of SYTs of shape
$\lambda$ over all partitions $\lambda$ of $n$ having at most three parts, filled with the numbers
$1,2,\ldots,n$, is equal to the
$n$th Motzkin number since
\[ \Psi\big(G_3(0,0)\big) = m(x).\]

Consider the case $k=2$ in $(\ref{eq:730})$, that is, $\mu=(1)$. Since the
Motzkin number generating function $m(x)$ satisfies the functional
identity \[ m(x)= 1+xm(x)+x^2m(x), \] we deduce
        \begin{align*}
           \Psi\big(G_3(1,0)\big) &= m(x)\big(1+xm(x) \big)   \\
            &= {m(x)-1 \over x}.
         \end{align*}
This tells us that the number of SYTs of skew shape
$\lambda / (1)$ over all partitions $\lambda$ of $n$ having at most three parts, filled with the numbers $1,2,\ldots,n-1$ is equal
to the $n$th Motzkin number for $n \ge 1$.

Now let us find the generating function for $\Psi \big(
G_3(k-1,0)\big)$ for $k \ge 1$. From $(\ref{eq:730})$,
we have
{\allowdisplaybreaks
        \begin{align}\label{eq:731}
         \sum_{k=1}^\infty ( \alpha_0 +\alpha_1 -\alpha_k -\alpha_{k+1} ) y^k
         &= \sum_{k=1}^\infty {m(x) \over 1-xm(x)}  (1-x^km(x)^k
         ) y^k \nonumber \\
         &= {m(x) \over 1-xm(x)}  \Big({1 \over 1-y}-{1 \over 1-xym(x)} \Big) \nonumber \\
         &= -{y^2 \over (1-y)(x+xy-y+xy^2)} \nonumber \\
         & \qquad + {xy \over (1-y)(x+xy-y+xy^2)}m(x),
         \end{align}
         }
where the last equation follows by rationalizing the
denominator. 
By factoring the denominator of the first expression in the right
side of equation $(\ref{eq:731})$, we have
        \begin{equation}\label{eq:732}
        -{y^2 \over (1-y)(x+xy-y+xy^2)}= -{y^2 \over x(1-y) (1+y-{y / x}+y^2)},
       \end{equation}
so when expanded in powers of $y$, the function
$(\ref{eq:732})$ has only negative powers of $x$. However, the left
side of equation $(\ref{eq:731})$ has no negative powers of
$x$. So we can conclude that the function $(\ref{eq:732})$ must be
cancelled with negative powers of $x$ from the second expression
in the right side of equation $(\ref{eq:731})$.
Therefore, we know that the generating function for $\Psi \big(
G_3(k-1)\big)$ is the part of
       \[ {xy \over
(1-y)(x+xy-y+xy^2)}m(x),
       \]
consisting of nonnegative powers of $x$. 

Let
        \begin{equation}\label{eq:733}
        \displaystyle R(x,y):= {xy \over
(1-y)(x+xy-y+xy^2)}.
       \end{equation}
By factoring \[ \displaystyle {xy \over
(x+xy-y+xy^2)}={y \over 1+(1-x^{-1})y+y^2},\] we can see that the coefficient of $y^n$ in $(\ref{eq:733})$ is a polynomial in $1/x$. So there are polynomials $r_k(x)=\sum_i r_{k,i}x^i$ such that the function $(\ref{eq:733})$ can be written as
        \begin{equation}\label{eq:7034}
        \displaystyle {y \over (1-y)\left(1+(1-x^{-1})y+y^2\right)}
        =\sum_{k=1}^\infty r_k\left({1 \over x}\right)y^k.
       \end{equation}
       
Now we need to compute the coefficient of $x^{n-(k-1)}y^k$ in the expression $R(x,y)m(x)$ to find the number of SYTs of skew shape $\lambda / (k-1)$ over all partitions $\lambda$ of $n$ having at most three parts, filled with the numbers $1,2,\ldots,n-(k-1)$. By $(\ref{eq:7034})$, we have 
        \begin{align*}
         [x^{n-k+1}y^k] R(x,y)m(x)
         &= [x^{n-k+1}] r_k\left({1 \over x}\right) m(x) \\
         & = [x^{n-k+1}] \sum_{i=0}^\infty M_i x^i r_k\left({1 \over x}\right) \\
         &= [x^{n-k+1}] \sum_{i,j=0}^\infty M_i x^i r_{k,j} x^{-j} \\
         &= \sum_{i} M_i r_{k,i-n+k-1}  \\
         &= \sum_{i} M_{i+n-k+1} r_{k,i}. 
         \end{align*}
So, we summarize with the following theorem.
\begin{theorem}\label{thm:71}
For $k \ge 1$, the number of standard Young tableaux of skew shape $\lambda / (k-1)$ over all partitions $\lambda$ of $n$ having at most three parts, filled with the numbers $1,2,\ldots,n-k+1$ can be written as the linear combination of Motzkin numbers
\[ \sum_{i}r_{k,i} M_{i+n-k+1},\]
where the number $M_n$ is the $n$th Motzkin number and the coefficients $r_{k,i}$ are defined by 
\[ \sum_{k=1}^\infty r_k(x)y^k= \sum_{k=1}^\infty \sum_{i=0}^{k-1} r_{k,i}x^iy^k= {y \over (1-y)\left(1+(1-x)y+y^2\right)}.\]
\end{theorem}
Table $\ref{T:71}$ shows the coefficients $r_{k,i}$ for
$k$ from $1$ to $8$.
{ \begin{table}
 \centering \begin{tabular}{c|crrrrrrrr}
        $k \setminus i$ & 0 & 1 & 2 & 3 & 4 & 5 & 6 & 7 \\
        \hline
        $0$ & $0$  &   &   &   &   &   &   &      \\
        $1$ & $1$  &   &   &   &   &   &   &     \\
        $2$ & $0$  & $1$  &   &   &  &  &  &    \\
        $3$ & $0$ & $-1$ & $1$ &  &  &  &  &    \\
        $4$ & $1$ & $0$ & $-2$ & $1$  &  &  &  &    \\
        $5$ & $0$  & $2$ & $1$  & $-3$ & $1$  &  &  &    \\
        $6$ & $0$  & $-2$  & $3$ & $3$ & $-4$  & $1$  &  &    \\
        $7$ & $1$  & $0$  & $-6$  & $3$  & $6$  & $-5$  & $1$  &    \\
        $8$ & $0$  & $3$  & $3$  & $-12$  & $1$  & $10$  & $-6$  & $1$\\
  \end{tabular}
\caption{The values of $r_{k,i}$}\label{T:71}
\end{table} }

\subsection{Generalization of Theorem $\ref{thm:71}$}
Now we want to consider a generalization of Theorem $\ref{thm:71}$; for a partition $\mu$ having at most three rows, we want to express the number of SYTs of skew shape $\lambda / \mu$ over all partitions $\lambda$ of $n$ having at most three parts, filled with the numbers
$1,2,\ldots, n- |\mu|$, as a linear combination of Motzkin numbers. It is enough to consider partitions $\mu$ having at most two rows.

First let us compute $\Psi(h f_k)$ for $k>0$. By equation
$(\ref{eq:729})$ we have
{\allowdisplaybreaks
        \begin{align}\label{eq:735}
            \Psi(h f_k) &= L\Big(\theta( h\big(g_0 +2(g_1+\cdots +g_{k-1})+g_k\big) \Big) \nonumber \\
            &=  {1 \over \sqrt{1-2x-3x^2}} \Big( 1+ 2 \big(xm(x)+\cdots+ x^{k-1}m(x)^{k-1} \big) +x^k m(x)^k \Big) \nonumber \\
            &= {1 \over \sqrt{1-2x-3x^2}}  {1-x^{k+1}m(x)^{k+1}+xm(x)-x^km(x)^k \over 1-xm(x)} \nonumber \\
            &= {m(x) \over 1-x^2m(x)^2}  {\big( 1 + xm(x)\big) \big(1-x^{k}m(x)^{k}\big)  \over 1-xm(x)} \nonumber \\
            &= {m(x) \over \big(1-xm(x)\big)^2} \big(1-x^{k}m(x)^{k}\big).
         \end{align}
}

Let $\beta_k:=\Psi(h f_k)$ for $k >0$. Then
by $(\ref{eq:735})$, we can compute the generating function $B(y)$
for $\beta_k$ as the following:
{\allowdisplaybreaks
        \begin{align}\label{eq:736}
         B(y) &=\sum_{k=1}^\infty \beta_ky^k 
		= \sum_{k=1}^\infty {m(x) \over \big(1-xm(x)\big)^2} \big(1-x^{k}m(x)^{k}\big)
         y^k \nonumber \\
         &= {m(x) \over \big(1-xm(x)\big)^2} \Big({1 \over 1-y}-{1 \over 1-xm(x)y} \Big) \nonumber \\
       &= {-y(x^2-x^2 y)m(x) \over (1-y)(1-3x)(x+xy-y+xy^2)}+ {y(x-y+2xy) \over (1-y)(1-3x)(x+xy-y+xy^2)},
         \end{align}
         }
         where the last equation follows by rationalizing the denominator.         
Then the generating function for $\Psi(h f_j)+\Psi(h f_k)-\Psi(h f_{j+k})$ is 
{\allowdisplaybreaks
        \begin{align}\label{eq:737}
         \sum_{j,k=1}^\infty \big(\Psi(h f_j)+ &  \Psi(h f_k)- \Psi(h f_{j+k})\big) y^jz^k \nonumber \\
         &= \sum_{j,k=1}^\infty \big( \beta_j + \beta_k - \beta_{j+k} \big) y^jz^k \nonumber \\
         &= {z \over 1-z}B(y)+ {y \over 1-y}B(z)- \sum_{j,k=1}^\infty \beta_{j+k}y^jz^k \nonumber \\
         &= {z \over 1-z}B(y)+ {y \over 1-y}B(z)- \sum_{n=1}^\infty \beta_{n} {y^nz-yz^n \over y-z} \nonumber \\
         &= {z \over 1-z}B(y)+ {y \over 1-y}B(z)- {1 \over y-z}\Big( zB(y)-yB(z)\Big).
         \end{align}
         }
So, plugging equation $(\ref{eq:736})$ into $(\ref{eq:737})$ gives the generating function for $\Psi(h
f_j)+\Psi(h f_k)-\Psi(h f_{j+k})$, which is equal to
    \begin{equation}\label{eq:738}
    \begin{split}
      {yz(1-yz)m(x) \over
        (1-y) \big(1+(1-1/x)y+y^2 \big) (1-z) \big( 1+(1-1/x)z+ z^2 \big)}  \qquad \qquad &\\
    - {yz(xyz+xy+xz-yz) \over
        (1-y) \left( x+xy-y+xy^2 \right) (1-z) \left( x+xz-z+xz^2 \right)}.& \\
    \end{split}
    \end{equation}

Let 
    \begin{equation}\label{eq:7040}
       T(x,y,z):= {yz(1-yz) \over
        (1-y) \big(1+(1-1/x)y+y^2 \big) (1-z) \big( 1+(1-1/x)z+ z^2 \big) },
       \end{equation}
       and let
           \begin{equation*}
       S(x,y):={y \over
        (1-y)\big( 1+(1-x)y+y^2 \big) }.
       \end{equation*}
       Then by equation $(\ref{eq:7034})$ we have
       \[ S(x,y)=\sum_{k=1}^\infty r_k(x)y^k .\] 
       
We know that when expanded in powers of $y$, the second term in $(\ref{eq:738})$ has only negative powers of $x$. However the left side of equation in $(\ref{eq:737})$ has no negative powers of $x$. So the first term in $(\ref{eq:738})$ must be cancelled with negative powers of $x$ from the second term in $(\ref{eq:738})$. Also, the expression $(\ref{eq:7040})$ has only negative powers of $x$.
Then we simplify the function
        \begin{align}\label{eq:7370}
         T(1/x,y,z)
         &= {yz(1-yz) \over (1-y) \big(1+(1-x)y+y^2\big)(1-z)\big(1+(1-x)z+z^2 \big)} \nonumber \\
         &= yz(1-yz) S(x,y)S(x,z) \nonumber \\
         &= yz(1-yz) \sum_{j=1}^\infty r_j(x)y^j \sum_{k=1}^\infty r_k(x)z^k  \nonumber \\
         &= \sum_{j,k=1}^\infty \left(r_j(x)r_k(x)-r_{j-1}(x)r_{k-1}(x)\right)y^jz^k,
         \end{align} where $r_0(x)=0$.
         
Define the polynomial $r_{j,k}(x)$ by 
 \begin{equation}\label{eq:7041}
T(1/x,y,z)=\sum_{j,k}r_{j,k}(x)y^jz^k.
  \end{equation} 
Then by (\ref{eq:7370}) and (\ref{eq:7041}) we have \[r_{j,k}(x)=r_j(x)r_k(x)-r_{j-1}(x)r_{k-1}(x).\]

Next, let us compute the coefficient of $x^{n-(|\mu_1|+|\mu_2|)}y^jz^k$ in the expression $T(x,y,z)m(x)$. By $(\ref{eq:7041})$, we have 
{\allowdisplaybreaks
        \begin{align*}
         [x^{n-(|\mu_1|+|\mu_2|)}y^jz^k] \, T(x,y,z)m(x)
         &= [x^{n-(|\mu_1|+|\mu_2|)}] \, r_{j,k}\left({1 \over x}\right) m(x) \\
         &= [x^{n-(|\mu_1|+|\mu_2|)}] \sum_{i=0}^\infty M_i x^i r_{j,k}\left({1 \over x}\right) \\
         &= [x^{n-(|\mu_1|+|\mu_2|)}] \sum_{i,l=0}^\infty M_i x^i r_{i,j,k} x^{-l} \\
         &= \sum_{i} M_i r_{i-n+( |\mu_1|+|\mu_2| ),j,k}  \\
         &= \sum_{i} M_{i+n-(|\mu_1|+|\mu_2|)} r_{i,j,k}. 
         \end{align*}
         }
So, we summarize with the following theorem.

\begin{theorem} 
For positive integers $\mu_1$ and $\mu_2$ with $\mu_1 \ge \mu_2$, the number of standard Young
tableaux of skew shape $\lambda / (\mu_1,\mu_2)$ over all partitions $\lambda$ of $n$ having at most three parts, filled with the numbers $1,2,\ldots,n-(\mu_1 +\mu_2)$, can be written as the linear combination of Motzkin numbers
\[\sum_i M_{i+n-(|\mu_1|+|\mu_2|)} r_{i,j,k} \]
where the number $M_n$ is the $n$th Motzkin number, $j=\mu_2 +1$, $k=\mu_1-\mu_2+1$, and  $r_{i,j,k}$ is the coefficient of $x^{i}$ in $r_{j,k}(x)=r_j(x)r_k(x)-r_{j-1}(x)r_{k-1}(x)$ defined in (\ref{eq:7041}).
\end{theorem}

\subsection{Focus on the polynomial $r_n(x)$} Now we study the polynomial $r_n(x)$. From $(\ref{eq:7034})$, we have  
       \begin{equation}\label{eq:7035}
       \sum_{n=0}^\infty r_n(x)y^n= {y \over (1-y)\left(1+(1-x)y+y^2\right)}.
       \end{equation}
       
Define polynomials $q_n(x)$ by
    \begin{equation}\label{eq:741}
       \sum_{n=0}^\infty q_n(x)z^n = {1 \over 1-xz+z^2}.
       \end{equation}
Then $q_n(x)=U_n(x/2)$ where $U_n(x)$ for all $n\geq0$ is the Chebyshev polynomial of the second kind (\seqnum{A093614}), which can be defined by the generating function
    \begin{equation*}
         \frac{1}{1-2xz+z^2} = \sum_{n=0}^{\infty}U_n(x)z^n.
    \end{equation*}
    
Now we are trying to find a formula for $r_n(x)$ in terms of the Chebyshev polynomial of the second kind. Then by $(\ref{eq:7035})$ we have
            \begin{equation}\label{eq:734}
       \sum_{n=0}^\infty r_n(x+1)y^n= {y \over (1-y)(1-xy+y^2)}.
       \end{equation}
So, from $(\ref{eq:741})$ and $(\ref{eq:734})$, we deduce 
           \begin{equation*}
       r_n(x+1) =\sum_{i=0}^{n-1} U_i(x/2).
       \end{equation*} 
Then we can rewrite $(\ref{eq:734})$ as the sum of an even function 
and an odd function of $x$:
        \begin{align*}
           { y  \over (1-y)(1-xy+y^2 )}
            &= { y(1+y^2)\over (1-y)(1-xy+y^2)(1+xy+y^2) } \\
            & \qquad +  { xy^2\over (1-y)(1-xy+y^2)(1+xy+y^2) }.
        \end{align*}
        
Let 
        \begin{align*}
           P_e(x,y) &:={ y(1+y^2) \over (1-y)(1-xy+y^2)(1+xy+y^2) } 
           \intertext{and}
          P_o(x,y) &:={ xy^2 \over  (1-y)(1-xy+y^2)(1+xy+y^2) }.
       \end{align*}
        
Then the following lemma shows that we can express $P_e(x,y)$ and $P_o(x,y)$ as the Chebyshev polynomial of the second kind.
\begin{lemma}\label{lemma:73} 
The even function $P_e(x,y)$ and the odd function $P_o(x,y)$ in $ y  / (1-y)(1-xy+y^2 )$ can be expressed as the Chebyshev polynomial of the second kind:
        {\allowdisplaybreaks
        \begin{align} \label{eq:7043}
        P_e(x,y) &= y(1+y) \sum_{n=0}^\infty U_n\left({x \over 2}\right)^2y^{2n},
         \intertext{and}
        P_o(x,y) &= (1+y) \sum_{n=1}^\infty U_n\left({x \over 2}\right) U_{n-1} \left({x \over 2}\right)y^{2n}.
           \end{align}
           }
\end{lemma}

\begin{proof} Let $i=\sqrt{-1}$.
Substituting $iy^2$ for $x$ and
$-xi $ for $a$ and $b$ in equation ($7$) in \cite{JH} gives us that
        \begin{align*}
           {P_e(x,y) \over y(1+y)} &= {1 \over 1-xy^2+y^4} * {1 \over
       1-xy^2+y^4}  \\
            &= \sum_{n=0}^\infty U_n\left({x \over 2}\right)y^{2n} * \sum_{n=0}^\infty U_n \left({x \over 2}\right)y^{2n} \nonumber \\
            &= \sum_{n=0}^\infty U_n\left({x \over 2}\right)^2y^{2n}, \nonumber
        \end{align*}
where $*$ denotes the Hadamard product in $y$. So we have equation (\ref{eq:7043}).

Similarly, the function $P_o(x,y)/(1+y)$ can be
written as 
        \begin{align*}
           {P_o(x,y) \over (1+y)} &= {1 \over 1-xy^2+y^4} * {y^2 \over
       1-xy^2+y^4}  \\
            &= \sum_{n=0}^\infty U_n\left({x \over 2}\right)y^{2n} * \sum_{n=0}^\infty U_n \left({x \over 2}\right)y^{2n+2} \nonumber \\
            &= \sum_{n=1}^\infty U_n\left({x \over 2}\right) U_{n-1} \left({x \over 2}\right)y^{2n}, \nonumber           
        \end{align*}
where $*$ denotes the Hadamard product in $y$.
\end{proof}

By Lemma $\ref{lemma:73}$ and $q_n(x)=U_n(x/2)$, the even function $P_{e}$ of $x$ satisfies that
        \begin{align*}
           P_{e}(x,y) &= y(1+y)\sum_{n=0}^\infty q_n(x)^2 y^{2n } \\
            &= q_0(x)^2y+q_0(x)^2y^2 + q_1(x)^2y^3+q_1(x)^2y^4 + \cdots \\
            &= \sum_{n=1}^\infty q_{\lfloor (n-1)/2
            \rfloor}(x)^2y^n ,
        \end{align*}
        where $q_n(x)$ is the polynomial in (\ref{eq:741}) for all $n \ge 0$.
Also the odd function can be expressed as
\[ P_o(x,y) =\sum_{n=2}^\infty q_{\lfloor n/2 \rfloor}(x)q_{\lfloor n/2 \rfloor-1}(x) y^n .\]
So the expression $(\ref{eq:734})$ can be written as
        \begin{equation}\label{eq:7045}
           { y  \over (1-y)(1-xy+y^2 )}
           = q_0(x)^2y+ \sum_{n=1}^\infty \big(q_{\lfloor (n-1)/2
            \rfloor}(x)^2 + q_{\lfloor n/2 \rfloor}(x)q_{\lfloor n/2 \rfloor-1}(x)\big)y^n .
        \end{equation} 
Therefore, equating coefficients in (\ref{eq:7045}) gives 
$r_1(x+1)=q_0(x)^2$ and 
\[r_n(x+1) = q_{\lfloor (n-1)/2
            \rfloor}(x)^2 + q_{\lfloor n/2 \rfloor}(x)q_{\lfloor n/2 \rfloor-1}(x), \textrm{ when } n \ge 2.\]
That is, for all $m>0$,
        \begin{align*}
           r_{2m}(x+1) &= q_{m-1}(x) \big( q_{m-1}(x)+q_{m}(x) \big),
           \intertext{and} 
           r_{2m+1}(x+1) &= q_{m}(x) \big( q_{m-1}(x)+q_{m}(x) \big).
        \end{align*}


%


\end{document}